\documentclass[reqno,twoside,final]{amsart}
\usepackage{amsmath,amstext,amsthm,amsfonts,amssymb,amscd}
\usepackage[numbers]{natbib}
\newcommand{\etype}[1]{\renewcommand{\labelenumi}{(#1{enumi})}}
\def\eroman{\etype{\roman}}
\newcommand{\C}{\mathbb{C}}

\newcommand{\Image}{{\operatorname{Im}\,}}

\newcommand{\Cent}{{\operatorname{Cent}}}
\newcommand{\tr}{{\operatorname{tr}}}

\newcommand{\ssl}{{\operatorname{sl}}}

\newcommand{\real}{{\operatorname{Re}\,}}
\newcommand{\ve}{{\operatorname{Ve}\,}}

\newcommand{\R}{\mathbb{R}}
\newcommand{\HH}{\mathbb{H}}

\theoremstyle{definition}
\newtheorem{definition}{Definition}

\newtheorem{rem}{Remark}
\newtheorem*{rem*}{Remark}
\newtheorem*{acknow*}{Acknowledgements}
\newtheorem*{examples*}{Examples}
\newtheorem{examples}{Example}
\theoremstyle{plain}

\newtheorem{lemma}{Lemma}

\newtheorem{theorem}{Theorem}
\newtheorem*{theorem*}{Theorem}
\newenvironment{proof-sketch}{\noindent{\bf Sketch of Proof}\hspace*{1em}}{\qed\bigskip}
\newenvironment{proof-idea}{\noindent{\bf Proof Idea}\hspace*{1em}}{\qed\bigskip}
\newenvironment{proof-of-lemma}[1]{\noindent{\bf Proof of Lemma #1}\hspace*{1em}}{\qed\bigskip}
\newenvironment{proof-of-prop}[1]{\noindent{\bf Proof of Proposition #1}\hspace*{1em}}{\qed\bigskip}
\newenvironment{proof-of-thm}[1]{\noindent{\bf Proof of Theorem #1}\hspace*{1em}}{\qed\bigskip}
\newenvironment{proof-attempt}{\noindent{\bf Proof Attempt}\hspace*{1em}}{\qed\bigskip}

\begin{document}

\title[Images of non-commutative polynomials]{The images of non-commutative polynomials evaluated on the Quaternion algebra}
\author{Sergey Malev}

\address{Department of mathematics, Ariel University} \email {sergeyma@ariel.ac.il}
\thanks{This research was supported by ERC Advanced grant Coimbra 320974.}
\thanks{The author is partially supported by Israeli Science Foundation grant no. 1623/16}
\thanks{The author was partially supported by Ariel University postdoctoral fellowship program}
\thanks{We would like to thank Agata Smoktunowicz, Danny Neftin and Alexei Kanel-Belov for interesting and fruitful discussions regarding this paper}

\maketitle

\begin{abstract}
%\subsection*{Abstract}
Let $p$ be a multilinear polynomial in several non-commuting
variables with coefficients in  an arbitrary field $K$.
Kaplansky
conjectured that for any $n$, the image of $p$ evaluated on the
set $M_n(K)$ of $n$ by $n$ matrices is a vector space.
In this paper we settle the analogous conjecture for a quaternion algebra.
\end{abstract}

%\today

%\maketitle
%\today
%\tableofcontents

\section{Introduction}
This note is the continuation of \cite{BMR1,BMR2,BMR3,BMR4,M}, in which
Kanel-Belov, Rowen and the author considered the question, reputedly
raised by Kaplansky, of the possible image set $\Image p$ of a
polynomial $p$ on matrices. (L'vov later reformulated this for
multilinear polynomials, asking whether $\Image p$ is a vector subspace.)
One can generalize this conjecture for an arbitrary simple finite dimensional algebra.
Note that for a non-simple finite dimensional algebra this conjecture can fail.
In particular, it fails for a Grassmann algebra, with a linear space of dimension more than or equal to $4$, and a field of
characteristic $\neq 2$. It is a finite dimensional (but not simple) associative algebra.
In this case one can consider a polynomial $p(x,y)=x\wedge y-y\wedge x$ which is multilinear.
In this case $e_1\wedge e_2=p(1/2e_1,e_2)$ and $e_3\wedge e_4$ both belong to the image of $p$, but their sum does not.
Thus the image is not a vector space.

 The main result in this note is for
the quaternion algebra,   
proving the following (see \S\ref{def1-q} for terminology):

\begin{theorem}
\label{main-q} If $p$ is a multilinear polynomial evaluated on
the quaternion algebra with Hamilton product $\HH(\R),$ then
$\Image p$ is either $\{0\}$, or  $\R\subseteq\HH(\R)$ (the space of scalar
quaternions), or $V\subseteq\HH(\R)$ -- the space of vector quaternions, or $\HH(\R)$. 
\end{theorem}

We also classify possible images of semi-homogeneous polynomials:
\begin{theorem}\label{homogen-q}
If $p$ is a semihomogeneous polynomial of weighted degree $d\neq 0$ evaluated on the
the quaternion algebra with Hamilton product $\HH(\R)$, then
$\Image p$ is either $\{0\}$,  or $\R$, or $\R_{\geq 0}$, or $\R_{\leq 0}$, or $V$, or some Zariski dense subset of $\HH$.
\end{theorem}

\section{Definitions and basic preliminaries}\label{def1-q}
\begin{definition}
By $K\langle x_1,\dots,x_m\rangle$ we
denote the free $K$-algebra generated by noncommuting variables
$x_1,\dots,x_m$, and refer to the elements of $K\langle
x_1,\dots,x_m\rangle$ as {\it polynomials}. Consider any algebra
$R$ over a field $K$. A polynomial $p\in K\langle
x_1,\dots,x_m\rangle$ is called a {\it polynomial identity} (PI)
of the algebra $R$ if $p(a_1,\dots,a_m)=0$ for all
$a_1,\dots,a_m\in R$;  $p\in K\langle x_1,\dots,x_m\rangle$ is a
{\it central polynomial} of $R$, if for any $a_1,\dots,a_m\in R$
one has $\mbox{$p(a_1,\dots,a_m)\in \Cent(R)$}$ but $p$ is not a
PI of $R$.
%The
%The
A polynomial $p$ (written as a sum of monomials) is called {\it
semi-homogeneous of weighted degree $d$} with (integer) {\it
weights} $(w_1,\dots,w_m)$  if for each monomial $h$ of $p$,
taking $d_j$ to be the degree of $x_{j}$ in $h$, we have
$$d_1w_1+\dots+d_mw_m=d.$$ A
semi-homogeneous polynomial with weights $(1,1,\dots, 1)$ is
called $\it{homogeneous}$ of degree $d$.

A polynomial $p$ is {\it completely homogeneous} of multidegree
$(d_1,\dots,d_m)$ if each variable $x_i$ appears the same number
of times $d_i$ in all monomials.
%The
A polynomial $p\in K\langle x_1,\dots,x_m\rangle$ is called {\it
multilinear} of degree $m$ if it is linear (i.e. completely homogeneous of
multidegree $(1,1,\dots,1)$). Thus, a polynomial is multilinear if
it is a polynomial of the form
$$p(x_1,\dots,x_m)=\sum_{\sigma\in S_m}c_\sigma
x_{\sigma(1)}\cdots x_{\sigma(m)},$$ where $S_m$ is the
%$m$-th
symmetric group in $m$ letters and the coefficients $c_\sigma$ are
constants in $K.$
\end{definition}
\begin{definition}
By quaternion algebra $\HH$ with Hamiltonian product we mean a four-dimensional algebra $\langle 1,i,j,k\rangle_\R$ such that
$$i^2=j^2=k^2=-1; ij=-ji=k; jk=-kj=i; ki=-ik=j.$$ In this algebra quaternions of $\R=\langle 1\rangle_\R$ are called scalars, and $V=\langle i,j,k\rangle_\R$ are called vectors. By basic quaternions we mean the following four elements of $\HH$: $\{1,i,j,k\}$. We also will use the standard quaternion functions: the norm $\left\vert\left\vert a+bi+cj+dk\right\vert\right\vert=\sqrt{a^2+b^2+c^2+d^2}$, the real part $\real (a+bi+cj+dk)=a$, the vector part $\ve (a+bi+cj+dk)=bi+cj+dk$.
\end{definition}
\begin{rem}
Any quaternion can be uniquely written as a sum of a scalar and a vector, hence the functions of real and vector parts are well defined. The function of norm is multiplicative.
\end{rem}
In \cite{A} Almutairi proved that if $p$ is a non-central multilinear polynomial then its image contains all vectors. In this note we provide a complete classification of the possible images.

\section{Multilinear case}
Let us start with proving the following not difficult but important lemmas:

\begin{lemma}\label{elem-q} Let $p$ be a multilinear polynomial.
If $a_i$ are basic quaternions, then   $p(a_1,\dots,a_m)$ is $c\cdot q$ for some basic quaternion $q$ and some scalar $c\in\R$ (in particular, in some cases $c$ can equal $0$).
\end{lemma}
\begin{proof}
Note that for any two basic quaternions $q_1$ and $q_2,$ $q_1q_2=\pm q_2q_1$. Therefore, taking products of $m$ basic quaternions we obtain the same result (up to $\pm$).
Thus the sum of these results multiplied by scalars must be an basic quaternion multiplied by some scalar coefficient.
\end{proof}
\begin{lemma}\label{cone-1-q}
For any multilinear polynomial $p$ $\Image p$ is a self similar cone, i.e. for any invertible $h\in\HH$ and any scalar $c\in\R$ and any element $\alpha\in\Image p$ $ch\alpha h^{-1}\in\Image p.$
\end{lemma}
\begin{proof}
If $p(x_1,\dots,x_m)=\sum_{\sigma\in S_m}c_\sigma
x_{\sigma(1)}\cdots x_{\sigma(m)},$ then
\begin{equation}
\begin{split}
p(hx_1h^{-1},hx_2h^{-1}\dots,hx_mh^{-1})=
\sum_{\sigma\in S_m}c_\sigma
hx_{\sigma(1)}h^{-1}\cdots hx_{\sigma(m)}h^{-1}=\\ =hp(x_1,\dots,x_m)h^{-1},
\end{split}
\end{equation}
 and
thus $p(chx_1h^{-1},hx_2h^{-1}\dots,hx_mh^{-1})=ch\alpha h^{-1}\in\Image p.$
\end{proof}
\begin{lemma}\label{cone-q} The set of vectors $V$ is an irreducible self-similar cone.
\end{lemma}
\begin{proof}
It is enough to show that any self-similar cone including element $i$ contains $V$. 
Take $h(y,z)=1+yj+zk$, thus $h^{-1}=\frac{1-yj-zk}{1+y^2+z^2}$.
Thus a minimal self-similar cone $C$ containing $i$ contains 
all elements $c\cdot hih^{-1}$, in particular it contains all elements $(1+yj+zk)i(1-yj-zk)=(1-y^2-z^2)i+2zj-2yk.$
Consider an arbitrary vector $ai+bj+ck$. If $b=c=0$ then this vector belongs to $C$ because it is a multiple of $i$. Assume that at least one of $b$ and $c$ is nonzero. Then $b^2+c^2>0$, hence one can take $l=\frac{a+\sqrt{a^2+b^2+c^2}}{2}$, 
$y=-\frac{c}{a+\sqrt{a^2+b^2+c^2}}$, $z=\frac{b}{a+\sqrt{a^2+b^2+c^2}}$, these numbers are well defined. Thus the element $(1+yj+zk)i(1-yj-zk)=(1-y^2-z^2)i+2zj-2yk=ai+bj+ck$ belongs to $C$.
\end{proof}

Now we are ready to prove the main theorem:
\medskip

\begin{proof-of-thm}{\ref{main-q}}
Let us substitute basic quaternions in the polynomial $p$. According to Lemma \ref{elem-q} as a result we will obtain multiples of basic quaternions. Consider four possible cases: all these results vanish, among these results there are scalars only, among these results there are vectors only, and among these results there are both, vectors and scalars.
The first two cases quickly lead to answers about the image of the polynomial $p$: in the first case $p$ is PI, its image is $\{0\}$, and in the second case it is central polynomial, and its image is $\R$.
In the third case its image is $V$, it follows from the Lemma \ref{cone-q}. Therefore the most interesting case is the fourth one.
We assume that there are basic quaternions 
$x_1,\dots,x_m$ and $y_1,\dots,y_m$ such that $p(x_1,\dots,x_m)=k\in\R\setminus\{0\}$ and $p(y_1,\dots,y_m)=v\in V\setminus\{0\}$. We will show that in this case image of $p$ is the set of all the quaternions. For that let us consider the following $m+1$ polynomials depending on $z_1,\dots,z_m$:
$A_1=p(x_1,x_2,\dots,x_m); A_2=p(z_1,x_2,\dots,x_m);
A_3=p(z_1,z_2,x_3,\dots,x_m);\dots; A_{m+1}=p(z_1,\dots,z_m)$.
Note that $A_1$ is a constant polynomial taking only one possible value (which is a nonzero scalar), for any $i\ \Image A_i\subseteq\Image A_{i+1}$ and $\Image A_{m+1}=\Image p$ includes nonscalar values. Therefore, there exists $i$ such that   $\Image A_i\subseteq\R$ and  $\Image A_{i
+1}\not\subseteq\R$. Thus there exist some collection of quaternions $r_1,r_2,\dots,r_m,r_i^*$ such that 
$p(r_1,r_2,\dots,r_m)=r\in\R\setminus\{0\}.$ and $p(r_1,r_2,\dots,r_{i-1}, r_i^*, r_{i+1},\dots,r_m)~\notin~\R$. Assume that $p(r_1,r_2,\dots,r_{i-1}, r_i^*, r_{i+1},\dots,r_m)=a+v$ for $a\in\R$ and $v\in V$. Then $v\neq 0$. If $a=c\cdot p(r_1,r_2,\dots,r_m)$ we can take 
$\tilde r_i=r_i^*-cr_i$, and 
$$p(r_1,r_2,\dots,r_{i-1}, \tilde r_i, r_{i+1},\dots,r_m)=v\in V\setminus\{0\}.$$
Note that for arbitrary real numbers $x$ and $y$ we have an element 
$$p(r_1,r_2,\dots,r_{i-1}, xr_i+yr_i\tilde r_i, r_{i+1},\dots,r_m)=xr+yv.$$ Consider an arbitrary quaternion $a+bi+cj+dk,$ where $a,b,c,d\in\R$. Let us take an $x$ such that $xr=a$. According to Lemma \ref{cone-q}, $V$ is an irreducible self-similar cone, thus  there exist $h\in\HH$ and $y\in\R$ such that $yhvh^{-1}=bi+cj+dk.$ Hence,  
$p(hr_1h^{-1},hr_2h^{-1},\dots,hr_{i-1}h^{-1},h(xr_i+y\tilde r_i)h^{-1},hr_{i+1}h^{-1},\dots,hr_mh^{-1}=
h(xr+yv)h^{-1}=xr+yhvh^{-1}=a+bi+cj+dk.$ Therefore, $\Image p=\HH$.
\end{proof-of-thm}

\medskip

\section{The homogeneous case}

%\begin{lemma}\label{clos-q}
%For any homogeneous polynomial $p(x_1,\dots,x_m)$ evaluated on $\HH$ $\Image p$ is closed with respect to standard real topology.
%\end{lemma}
%\begin{proof}
%Consider an Image of $p$ evaluated on tubes $(x_1,\dots,x_m)$ such that $\left\vert\left\vert x_i\right\vert\right\vert=1$ for all $i$. Then it is a continuous image of the compact space, since each $x_i$ belongs to the $3$-dimensional sphere in $\HH$ considered as $\R^4$, which is a compact tolopogical space, the product of compact spaces is a compact space, and the continuous image of the compact space is compact. Thus $I$, the image of $p$ evaluated on a set of all such tubes is a compact space. Note that image of $p$ evaluated on all possible tubes $(x_1,\dots,x_m)$ is a cone generated by $I$, therefore it must be a closed set.
%\end{proof}

Now we can consider the semihomogeneous case and present the proof of Theorem \ref{homogen-q}:

\begin{proof-of-thm}{\ref{homogen-q}}

We will use a famous fact that there exists an isomorphism 
$$\Phi: \HH(\R)\otimes \C_{\R}\rightarrow M_2(\C)_{\R}$$ such that for any $q=a+bi+cj+dk\in\HH$ and any $z\in\C$
$$\Phi(q\otimes z)= z\cdot
\begin{bmatrix}
    a+bi       & c+di  \\
    -c+di      & a-bi
\end{bmatrix}.$$

Note that $\Phi(\HH\otimes 1_\C)$ i.e. the set of matrices 
$$\begin{bmatrix}
    a+bi       & c+di  \\
    -c+di      & a-bi
\end{bmatrix}$$
is Zariski dense in $M_2(\C)$. In \cite[Theorem 1]{BMR1} we classified possible images of semihomogeneous polynomial evaluated on $M_2(K)$, where $K$ is arbitrary quadratically closed field. We showed that the following images may occur: 
$\Image p$ is either $\{0\}$, $K$, the set of all non-nilpotent matrices having trace zero,
$\ssl_2 (K)$, or a dense subset of $M_2 (K)$ (with respect to Zariski topology). Therefore, the Zariski closure of the image of polynomial evaluated on $2\times 2$ matrices with complex entries must be either $\{0\}$, or $\C$, or $\ssl_2 (\C)$ or $M_2 (\C)$. Thus an image of polynomial evaluated on quaternions is Zariski dense in $\{0\}$, or $\R$, or $V$ or $\HH$.
Note that if $p$ is a semihomogeneous polynomial, then there exist numbers $w_1,\dots,w_m$ and $d\neq 0$ such that for any $c\in\R$ 
$p(c^{w_1}x_1,c^{w_2}x_2,\dots,c^{w_m}x_m)=c^dp(x_1,\dots,x_m)$. Therefore $\Image p$ must be a cone with respect to positive real multipliers, i.e. for any $x\in\Image p$ and any $\lambda>0,$ $\lambda x\in\Image p$. Let us consider four possible cases: the image is Zariski dense in $\{0\}$, in $\R$, in $V$, and in $\HH$.
In the first case $p$ is PI, and its image is $\{0\}$. In the second case $p$ is a central polynomial, therefore an image of $p$ being a cone with respect to positive multipliers should be either $\R$, or $\R_{\geq 0}$, or $\R_{\leq 0}$.
In the third case, $p$ takes only vector values. According to Lemma \ref{cone-q} we know that $V$ is an irreducible self-similar cone (up to real multipliers), and $\Image p$ is a self-similar cone up to positive real multipliers, therefore $\Image p\cup  (-\Image p)=V,$ i.e. for any vector $v\in V$ either $v$, or $-v$ belongs to $\Image p$. Without loss of generality, assume that $i\in\Image p$. Hence $jij^{-1}\in\Image p$, and $jij^{-1}=-i$. Hence for any $c\in\R$ $ci\in V$, and thus any element $v\in V$ is conjugated with some $ci$, as we showed in the proof of Lemma \ref{cone-q}. Therefore, $\Image p=V$. In the forth case an image of $p$ is a Zariski dense subset of $\HH$.
\end{proof-of-thm}

Let us show some interesting examples of multilinear and homogeneous polynomials:
\begin{examples}
\begin{enumerate}\eroman  
\item A multilinear polynomial can be PI (according to Amitsur-Levitsky Theorem $s_4$ evaluated on $2\times 2$ matrices is a PI, and thus is PI on quaternions).
\item A Lie bracket $p(x,y)=xy-yx$ is a vector-valued polynomial, and its image is $V$.
\item a square of any vector is a scalar, thus the polymomial 
$p(x,y)=xy+yx=(x+y)^2-x^2-y^2$ evaluated on vectors $V$ takes only scalar values, therefore the polynomial $p(x_1,x_2,x_3,x_4)=[x_1,x_2][x_3,x_4]+[x_3,x_4][x_1,x_2]$ evaluated on $\HH$ is a multilinear central polynomial.
\item The image of the polynomial $p(x)=x$ is the set of all quaternions.
\item The polynomial $p(x,y)=[x,y]^2$ is central, and in \cite{M} in order to provide an example of complete homogeneous polynomial evaluated of $2\times 2$ matrices with real entries and taking only positive central values I took a square of $p$. Note that evaluated on quaternions we do not need to take a square. This polynomial takes only non positive values: indeed, $(ai+bj+ck)^2=-a^2-b^2-c^2$, and $\Image p$ is $\R_{\leq 0}$. Hence, $-p(x,y)=-[x,y]^2$ is an example of the polynomial with image set $\R_{\geq 0}$.

\item The polynomial $p(x,y)=[x,y]^2+[x^2,y^2]$ is the sum of two polynomials, the image of the first term is $\R_{\leq 0}$, and the second one has image $V$, thus image of $p$ is the set of quaternions with non-positive real part. Of course $-p$ has an opposite image: the set of quaternions with non-negative real part.

\item One of the examples (see \cite[Example 4(i)]{BMR1}) is suitable for our case: the polynomial $g(x_1 , x_2 ) = [x_1 , x_2 ]^2$ has the property that $g(A, B) = 0$
whenever $A$ is scalar, but $g$ can take on a non-zero value whenever $A$ is non-scalar. Thus, $g(x_1 , x_2 )x_1$ takes all values except scalars.

\item Recall that any quaternion $q=a+v, a\in\R, v\in V$ can be considered as a $2\times 2$ matrix with complex entries, its eigenvalues are $\lambda_{1,2}=a\pm ni,$ where $n=\left\vert\left\vert v\right\vert\right\vert$ is the norm of the vector part of the quaternion. In particular 
$\tr\, q=2a=2\real q$ and $\det q=a^2+n^2=\left\vert\left\vert q\right\vert\right\vert$. In \cite[Example 4(ii)]{BMR1} we provided an example of the polynomial taking all possible values except for those where the ratio of eigenvalues belongs to some set $S$. However here we should have a polynomial with real coefficients. Nevertheless this is possible: 
let $S$ be any finite subset of $S^2$ -- a unit circle on the complex plane. There exists a
completely homogeneous polynomial $p$ such that $\Image p$ is the
set of all quaternions  except the quaternions with ratio of
eigenvalues from $S$. Any $c\in S$ can be written in the form $c=\frac{a+bi}{a-bi}$ for some $a,b\in\R$. The construction of the polynomial is as follows. Consider
$$f(x)=x\cdot\prod_{c\in
S}((a+bi)\lambda_1-(a-bi)\lambda_2)((a+bi)\lambda_2-(a-bi)\lambda_1),$$ where
$\lambda_{1,2}$ are eigenvalues of $x$ and $c=\frac{a+bi}{a-bi}$. For each $c$ the
product $(a+bi)\lambda_1-(a-bi)\lambda_2)((a+bi)\lambda_2-(a-bi)\lambda_1)=-a^2(\lambda_1-\lambda_2)^2-b^2(\lambda_1+\lambda_2)^2$
is
%symmetric polynomial of $\lambda_1$ and $\lambda_2$.
%Hence it is
a polynomial with real coefficients in $\tr\ x$ and $\tr\ x^2$. Thus $f(x)$ is a
polynomial with traces, and by \cite[Theorem
1.4.12]{Row}, one can rewrite each trace in $f$ as a fraction of
multilinear central polynomials (see \cite[equation (3) in Proposition 1]{BMR1} for details).
%of more than one variable, where the denominator is central.
After that we multiply the expression by the product of all the
denominators, which we can take to have value 1. We obtain a
completely homogeneous polynomial $p$ which image is the cone
under $\Image f$ and thus equals $\Image f$. The image of $p$ is
the set of all quaternions with ratios of eigenvalues
not belonging to $S$.
\end{enumerate}
\begin{rem}
A problem of classification of all possible Zariski dense images of semihomogeneous polynomials evaluated on $\HH$ remains being open.
\end{rem}
\end{examples}

\end{document}